\documentclass[11pt, reqno]{amsart}
\usepackage{amsthm}
\usepackage{mathtools,mathrsfs,amssymb,latexsym,graphics,enumerate}
\usepackage[mathscr]{eucal}
\usepackage{amsmath,amsfonts,amsthm,amssymb, cite}
\usepackage[left=1in,right=1in,top=1in,bottom=1in]{geometry}
\usepackage{graphicx}
\usepackage[colorlinks=true,
            linkcolor=red,
            urlcolor=blue,
            citecolor=red]{hyperref}

\renewcommand{\geq}{\geqslant}
\renewcommand{\leq}{\leqslant}

\numberwithin{equation}{section}

\newcommand{\rr}{\mathbb{R}}

\newcommand{\be}{\begin{eqnarray*}}
\newcommand{\bel}{\begin{eqnarray}}
\newcommand{\ee}{\end{eqnarray*}}
\newcommand{\eel}{\end{eqnarray}}
\newcommand{\ba}{\begin{aligned}}
\newcommand{\ea}{\end{aligned}}
\newcommand{\de}{\Delta}
\newcommand{\al}{\alpha}
\newcommand{\na}{\nabla}
\newcommand{\bu}{{\mathbf u}}
\newcommand{\ep}{\epsilon}

\newcommand{\pa}{\partial}
\newcommand{\wh}{\widehat}

\newcommand{\bx}{\mathbf{x}}

\mathtoolsset{showonlyrefs=true}

\newcommand{\myr}[1]{{\color{red}{#1}}} 
\newcommand{\myb}[1]{{\color{blue}{#1}}}
\newtheorem{thm}{Theorem}[section]
\newtheorem{cor}{Corollary}[section]

\newtheorem{lem}{Lemma}[section]

\theoremstyle{remark}

\newcommand{\grad}{\nabla}

\newcommand{\y}{x_2}

\newcommand{\bxm}{{\bx_-}}
\title{On the fast spreading scenario}
\date{\today}

\begin{document}

\title[On the fast spreading scenario]{On the fast spreading scenario}

\author{Siming He}
\address{Department of Mathematics, Duke University, Durham, NC 27708}
\email{siming.he@duke.edu}

\author{Eitan Tadmor}
\address{Department of Mathematics and Institute for Physical Science \& Technology\newline
 University of Maryland, College Park, MD 20742}
\email{tadmor@umd.edu}

\author{Andrej Zlato\v{s}}
\address{Department of Mathematics, University of California San Diego, La Jolla, CA 92093}
\email{zlatos@ucsd.edu}

\thanks{\textbf{Acknowledgment.} We thank  J.~Carrillo, F.~Hoffman, A.~Kiselev, and L.~Rhyzhik for useful discussions. This work was supported in part by NSF grants  DMS-2006372 and DMS-2006660 (SH);  NSF grant DMS-1613911 and ONR grant N00014-1812465 (ET); NSF grant DMS-1900943 and a Simons Fellowship (AZ)}
\subjclass{92C17; 80A25; 35K57; 35K15; 35B44}

\keywords{}


\begin{abstract}
We study  two types of divergence-free fluid flows on unbounded domains in two and three dimensions --- hyperbolic  and shear flows --- and their influence on chemotaxis and combustion.   We show that fast spreading by these flows, when they are strong enough, can suppress growth of solutions to PDE modeling these phenomena.  This includes prevention of singularity formation and global regularity of solutions to advective Patlak-Keller-Segel equations on $\rr^2$ and $\rr^3$, 
confirming numerical observations in \cite{KJCY16}, as well as quenching in advection-reaction-diffusion equations. 
\end{abstract}

\maketitle


\section{Introduction and Main Results}\ifx\myb{So I had a look at the proofs now.  It took a bit longer than I
thought because a couple of days after we posted the paper I finished
a month ago, an annoying issue cropped up that took a while to deal
with.  Then I went for a family trip in the mountains, so I didn’t
start looking at this until a few days ago.  Sorry for the delay.

Anyway, I’d like to make some edits in the paper, but before I do
them, could you first make a few adjustments?  You can then send me
the latest tex file and I’ll put in my edits.

In Step \#1 of the proof of Thm 1.1, it would be nice to add some more
details and citations so the reader doesn’t have to do come up with
computations himself.  Specifically:

1.  Citation or indication for the “standard calculation” leading to (2.18). \myr{A: I have added the complete calculation. And I have added a reference to the mass M.}

2. Indication of the estimate on the nonlinear term used in the
displayed line between (2.18) and (2.19). \myr{A: Added. Right now $T_1$ is the nonlinearity.}

3.  Wouldn’t it be better to show that the advective term integrates
to zero before integrating in time to get (2.19)? \myr{A: Changed.}

4.  It would be good to either include footnote 3 in Thm B.1 or
indicate a bit more the argument (or give a citation)? \myr{A: I have rewritten the Step 1 to make sure that the footnote is addressed properly}

5.  In the proof that the advective term integrates to 0, do we need
to pick $R_i$ in some way, since the moment bound is only stated as an
integral over the whole space. \myr{Here I just choose it to approach infinity such that the boundary contribution goes to zero. I apply the Lebesgue dominated convergence theorem  on a sequence of functions $\mathbf{1}_{|\bx|\leq R_i}n \bu\cdot \na n$ here. }

6.  Citation for the Moser-Alikakos iteration at the end of step 1. \myr{Added.}

Also, in Step 2:

7.  Explain what is $M$, it suddenly appears out of nowhere. \myr{Added a reference.}

8.  The displayed equation after (2.21) may need a bit more detail. \myr{Added more details.}

Btw, do we need a higher moment bound in the 3D argument of Thm 1.1? \myr{The moment calculation in 3D is actually different from before. I have added the calculations here. You can still derive higher moment estimates. When we do the explicit $L^2$-energy estimate, the second moment bound should be enough to guarantee the convergence of the boundary terms. Because we can follow the same strategy as in the two dimensional case, and the boundary terms near infinity will vanish if $||n||_\infty$  is qualitatively bounded and $\lim_{i\rightarrow\infty}\int_{\pa B(0,R_i)}|\bx|n dS_{R_i}$ to vanish.}

Also, the reaction-diffusion argument can be made much easier.  The
non-linear solution can be estimated by $e^t$ times the linear solution
(or $e^{ct}$ times it if $f(z)<=cz$).  So if you make the linear solution
small quickly, you immediately get the same for the nonlinear one.
Unless I missed something, this shows that there’s no need for
estimation of moments etc.
\myr{You are right! I adjust the argument accordingly and added the three dimensional version.}

I didn’t yet look at the shear flow stuff but that should just be a variation on the same theme, right? \myr{A: The shear case is easier and follows the same theme.}
}\fi

In this paper we study a \emph{fast spreading scenario}, a  dispersive phenomenon in which a density $n\geq 0$, subject to
advection, diffusion, and a nonlinear interaction $N$ modeled by the PDE
\begin{equation}\label{general}
 \pa_t n + A {\mathbf u}(\bx)\cdot \na n =\de n+\mathcal{N}(n),
\end{equation}
undergoes rapid stretching in a preferred direction dictated by the incompressible advecting velocity field ${\mathbf u}$ with a large amplitude $A\gg1$.
We shall consider this setup in two as well as three spatial dimensions, with the following two types  of (divergence-free) advection regimes.
 
\medskip\noindent
\textbf{(A1) Hyperbolic flow regime:} ${\mathbf u}(\bx):=(-\bxm,x_d)$ on $\rr^d$, with  $\bx_-:=\frac{1}{d-1}(x_1,\ldots, x_{d-1})$ and hence advective stretching in the $x_d$-direction.

\medskip\noindent
\textbf{(A2) Shear flow regime:} $\mathbf{u}(\mathbf{x}):=(u(\bx_-),{\mathbf 0}_-)$ on $\rr\times\mathbb T^{d-1}$,  with 
$\bx_-:=(x_2,\ldots,x_d)$ and hence advective stretching  in the $x_1$-direction.


\medskip\noindent
We chose the stretching directions in the two cases above so that they agree  with classical  literature on hyperbolic and shear flows.
We will now consider the following two types of nonlinearities, the motivation for which we discuss in more detail below.

\medskip\noindent
\textbf{(N1) Patlak-Keller-Segel (PKS) type nonlinearity:}  $\mathcal{N}(n):= -\na\cdot(n\na(-\de)^{-1}n)$, modeling cell aggregation driven by the gradient of the concentration of a cell-excreted chemoattractant. 

\medskip\noindent
\textbf{(N2) Ignition type nonlinearity:} $\mathcal{N}(n):= f(n)$, with a Lipschitz reaction function  $f\geq 0$ such that $f>0$ on $(\alpha,1)$ for some ignition temperature $\alpha\in(0,1)$, modeling combustion.

\medskip
When $A=0$,  the non-linearity $\textbf{(N1)}$ is strong enough to enforce finite time blow-up of non-negative solutions, while $\textbf{(N2)}$ results in $\liminf_{t\to\infty} n(t,\bx)\geq 1$ locally uniformly for all initially large enough non-negative solutions (see below).
While the added large incompressible advection term $A\bu(\bx)\cdot \na n$ is ``neutral'' with respect to $L^p$-energy balance for any $p\geq 1$, the resulting fast spreading can still act against these effects of the nonlinearity 
in both cases and result in strong damping of the uniform norm of solutions. 
We will show here that this indeed happens for the various combinations of the flows $\textbf{(A1)}, \textbf{(A2)}$ and nonlinearities $\textbf{(N1)}, \textbf{(N2)}$ in two and three dimensions, yielding singularity prevention for $\textbf{(N1)}$ and quenching for $\textbf{(N2)}$.

 
\subsection{Hyperbolic Flow Regime}
We will first consider \eqref{general} with the flow $\textbf{(A1)}$.
It is natural to study this model by rescaling time to obtain
\begin{align} \label{1.1a}
\pa_t n+ (-\bxm,x_d)\cdot \na n= A^{-1}\de n+ A^{-1}\mathcal{N}(n),
\end{align}
and then consider this PDE  as a perturbation of the linear equation
\[
\pa_t \rho+(-\bxm,x_d)\cdot \na \rho=A^{-1}\de \rho
\]
at times $t\ll A$ (i.e., $t\ll 1$ before rescaling).
The task is now to control the size of this perturbation, to which end one can try to employ energy-type arguments, as well as to demonstrate sufficiently fast decay for the passive scalar dynamics of $\rho$ for times $t\ll A$ (both  for large enough $A$), and then iterate this process in order to encompass larger times as well. 
We will do so here for both nonlinearities  ${\bf (N1)}$ and ${\bf (N2)}$.

\medskip\noindent 
\textbf{Case $\mathbf{(A1)+(N1)}$.}
This is the $d$-dimensional Patlak-Keller-Segel model with hyperbolic flow
\begin{equation}\label{PKS}
\pa_t n+A(-\bxm,x_d)\cdot\na n =\de n - \na \cdot(n\na (-\Delta)^{-1} n), 
\qquad \bx=((d-1)\bx_-,x_d)\in\rr^d, 
\end{equation}
with initial data $n(0,\bx)=n_0(\bx)\geq 0$ and $d\in\{2,3\}$.
Such equations model chemotaxis phenomena of biological organisms in fluid streams, with $n$ and $c:=(-\Delta)^{-1}n$ corresponding to the densities of cells and chemoattractants released by them, respectively. The cell density is subject to diffusion and chemical-triggered aggregation, as well as advection by the underlying fluid flow.  It is also assumed here that the chemoattractant diffuses much faster  than the cells themselves, which motivates the relationship $c=(-\Delta)^{-1}n$. 
We note that the \eqref{PKS} is one of many models incorporating the fluid advection effect (see, e.g., \cite{Lorz10,Lorz12,LiuLorz11,DuanLorzMarkowich10,FrancescoLorzMarkowich10,TaoWinkler,Tuval05,ChaeKangLee13,KozonoMiuraSugiyama}). 

In the absence of advection, \eqref{PKS} is the classical Patlak-Keller-Segel (PKS) model \cite{Patlak,KS}, whose variations have received considerable attention in mathematics and science. Since the literature on the subject is vast, we cannot list all of it here and instead refer the interested reader to the review \cite{Hortsmann} and some of the representative works \cite{ChildressPercus81,JagerLuckhaus92,Nagai95,Biler95,
HerreroVelazquez96,Biler06,BlanchetEJDE06,BlanchetCarrilloMasmoudi08,BlanchetCalvezCarrillo08,BlanchetCarlenCarrillo10} as well as references therein. In two spatial dimensions $d=2$, it is well-known that if the initial mass $\|n_0\|_1$ is below $8\pi$, then \eqref{PKS} with $A=0$ admits global smooth solutions. Above this threshold aggregation dominates diffusion, resulting in a Dirac-type singularity formation in finite time \cite{BlanchetEJDE06,JagerLuckhaus92}; and at the borderline, solutions with finite second moments remain regular and form Dirac masses as time approaches infinity \cite{BlanchetCarrilloMasmoudi08}.  On the other hand, when $d=3$ and $A=0$, 
solutions with  $\|n_0\|_{3/2}$  small enough are known to exist globally \cite{Perthame07,CorriasPerthameZaag04}, while finite time blow-up can occur for solutions with arbitrarily small masses \cite{CorriasPerthameZaag04}.

The first work studying blow-up suppression in chemotaxis models by strong flows is the paper \cite{KiselevXu15} by  Kiselev-Xu, where they applied strong relaxation enhancing flows  introduced by Constantin-Kiselev-Ryzhik-Zlato\v s in  \cite{CKRZ08}  to suppress  singularity formation in \eqref{PKS} on $\mathbb T^d$ ($d\in\{2,3\}$). This was later generalized to the fractional diffusion setting in \cite{HopfRodrigo18}.  Singularity suppression in PKS models by strong shear flows on the same domains was studied by Bedrossian-He \cite{BedrossianHe16} and He \cite{He}, with $\|n_0\|_1<8\pi$ needed for positive results when $d=3$, while Iyer-Xu-Zlato\v{s} obtained similar results without the limitation on $\|n_0\|_1$ for cellular and other flows that sufficiently enhance diffusion on $\mathbb T^d$ and on $[0,1]^d$ in \cite{IyerXuZlatos}. 

These results all hold on tori, and most (except for \cite{BedrossianHe16,He}) rely on strong mixing properties of the corresponding flows.  The flows $\mathbf{(A1)}$ and $\mathbf{(A2)}$ are both very basic and classical examples of advective motions but neither has such properties (this is also the reason for the hypothesis $\|n_0\|_1<8\pi$ in \cite{BedrossianHe16} when $d=3$).  Moreover, one cannot extend the above results to unbounded domains by simple periodization because \eqref{PKS} does not have a maximum principle.  

In fact, the only prior result on $\rr^2$ is in the paper \cite{HeTadmor172} by  He-Tadmor, where the fast-splitting hyperbolic flow $\mathbf{(A1)}$ was shown to suppress blow-up in \eqref{PKS} for solutions with $\|n_0\|_1<16\pi$ as well as additional structural assumptions. 
We substantially improve here this result by removing both these limitations, and also  extend it to $\rr^3$.  In particular, the flow $\mathbf{(A1)}$ prevents finite time singularity formation for solutions with arbitrarily large masses in dimensions two and three when its amplitude $A$ is sufficiently large.  

\begin{thm}[{\bf PKS + hyperbolic flow}]\label{Thm:PKS_suppression_of_blow_up}
Let $n$ solve \eqref{PKS} with  initial data $0\leq n_0\in L^1(\rr^d)\cap H^2(\rr^d)$ satisfying $|\bx^2|n_0(\bx)\in L^1(\rr^d)$ and $d\in\{2,3\}$. There is  $A_0=A_0(||n_{0}||_{L^1\cap H^2})$ such that if $A\geq A_0$, then $\sup_{t\geq 0} \|n(t,\cdot)\|_{L^\infty}<\infty$. 
\end{thm}


We also prove asymptotic decay of solutions  in three dimensions.  

\begin{thm}
\label{Thm:PKS_long_time_behavior}
When $d=3$, in Theorem \ref{Thm:PKS_suppression_of_blow_up} we also have $ \| n(t,\cdot)\|_{L^2}\leq t^{-1/2}$ for all $A\geq A_0$ and $t\geq 1$.  In fact, for these $A$ and $t$ we have $\int_{B(0,t^{1/4})} n(t,\bx)d\bx \leq t^{-1/8}$, showing that asymptotically in time, all the mass of the solution will escape to infinity.
\end{thm} 

We note that one can obtain the same results for the flow $-\mathbf{u}$ (which equals $\mathbf{u}$ up to rotation in 2D but not in 3D), by proving an analog of  Corollary \ref{Cor:fast-decay} below for this flow.


\medskip
\noindent
{\bf Remark.}  In the above results, the key property of the hyperbolic flow is that its {\it dissipation time} converges to 0 as $A\to\infty$ (this is the same for the shear flows with no plateaus discussed below).  Similarly to \cite{IyerXuZlatos}, we define the {\it dissipation time} of a divergence-free (time-independent) flow ${\bf u}$ to be the smallest $\tau> 0$ (which we  denote by $\tau({\bf u})$) such that the solution operator $\mathcal S_{t}^{\bf u}$ for the PDE
\[
\pa_t \rho+{\bf u}(\bx) \cdot \na \rho=\de \rho
\]
satisfies $||\mathcal S_{\tau}^{\bf u}||_{L^1\to L^\infty}\leq\frac 12$ (in \cite{IyerXuZlatos}, $||\mathcal S_{\tau}^{\bf u}||_{L^2\to L^2}$ was used instead).  Our proofs of the above results then show that they hold for any divergence-free flow ${\bf u}$ in place of $(-{\bf x}_-,x_d)$ that grows at most linearly as $|\bx|\to\infty$ and satisfies Theorem \ref{thm:local_Hs_0} below (local well-posedness for \eqref{PKS}) as well as $\lim_{A\to\infty} \tau(A{\bf u}) = 0$ (the latter holds for the hyperbolic flow by Corollary \ref{Cor:fast-decay}).  Indeed, one then only needs to replace $10\log A$ by $A\tau(A{\bf u})$ in the proof of Theorem \ref{Thm:PKS_suppression_of_blow_up} (which concerns the rescaled problem \eqref{1.1a}, so this is $\tau(A{\bf u})$ before temporal scaling), because then one can use the definition of $\tau(A{\bf u})$ instead of Corollary \ref{Cor:fast-decay} to obtain the crucial estimate \eqref{2.222}.

\medskip\noindent 
\textbf{Case $\mathbf{(A1)+(N2)}$.}
This is the  advection-reaction-diffusion equation
\begin{equation}\label{Reaction-Diffusion}
\pa_t n+A(-\bxm,\y)\cdot \na n=\de n+ f(n), 
\qquad \bx=((d-1)\bx_-,x_d)\in\rr^d,
\end{equation}
with 
$n(0,\bx)=n_0(\bx)\geq 0$ and $d\in\{2,3\}$.
Since the foundational papers of Kolmogorov-Petrovskii-Piskunov \cite{KPP} and Fisher \cite{Fisher},  reaction-diffusion equations have been used to model multiple phenomena, including combustion, chemical kinetics, population dynamics, morphogenesis, and nerve pulse propagation.  It would again be impossible to list here all the relevant works and we only mention a small cross section of the vast literature \cite{Turing52,Berestycki02,Fife79,Murray02,Murray03,ShigesadaKawasaki97,HodgkinHuxley52,ChoLevy17,AW, BH,BH3,GF, Weinberger,Xin00, Xin09, ZlaInhomog3d}.
 
 Some of the above phenomena occur inside a fluid medium, which may be subject to advection that is either imposed (passive) or generated by the reactive process (active).  This may be the case, in particular, for population dynamics in rivers and seas, as well as for various combustion phenomena such as burning in internal combustion engines or nuclear reactions in stars.  The effects of such advection on the reactive process are two-fold and in a sense opposite.
On the one hand, the fluid flow may enhance the spread of substances to distant places and hence increase the influence range of the reactive process. On the other hand, this spreading effectively reduces density of the substance and may lessen the mean impact of the reaction, possibly up to the point of quenching it. The interplay of these two effects may have a profound impact on the reactive dynamics, particularly if the advection is strong, and this has recently been studied in many works including \cite{ABP, BHN-2, CKR00, CKRZ08, FKR, KS00, KiselevZlatos,ZlaPercol, ZlaSpeedup}. 

We consider here the effect of strong hyperbolic flows  on ignition reactions $\mathbf{(N2)}$, modeling combustion in flammable substances.  In this setting, 
$n$ in \eqref{Reaction-Diffusion} is the temperature of the substance, and
the reaction function $f$ models the combustive process, with burning only occurring above the ignition temperature $\alpha$.
While we will only assume that $n\geq 0$, often considers solutions $0\leq n\leq 1$, with 1 being the (normalized) maximal possible temperature, which motivates letting $f\equiv 0$ on $(-\infty,0]\cup [1,\infty)$.  The latter of course guarantees that $0\leq n\leq 1$ whenever $0\leq n_0\leq 1$.
It is then well-known \cite{Kanel, AW} that when $A=0$, solutions that are initially greater than $\alpha+\epsilon$ on a ball with radius $R_\epsilon$ will converge locally uniformly to 1 as $t\to\infty$ (for each $\epsilon>0$ and some $R_\epsilon$).  That is, sufficiently extended initial flames will propagate.

However, this picture may change dramatically in the presence of strong flows. In fact, Constantin-Kiselev-Ryzhik showed in \cite{CKR00}   that strong shear flows profiles on $\rr\times\mathbb T$ that do not have plateaus (i.e., regions where they are constant) can quench arbitrarily large (compactly supported) initial temperatures, provided the flow amplitude is large enough (this was later extended to shear flows with small enough plateaus by Kiselev-Zlato\v s \cite{KiselevZlatos}, as well as to periodic cellular flows on $\rr\times\mathbb T$ and $\rr^2$ by Fannjiang-Kiselev-Ryzhik \cite{FKR} and Zlato\v s \cite{ZlaSpeedup}).
 By quenching we mean that $\|n(t_0,\cdot)\|_\infty$ drops below $\alpha$ at some time $t_0$, after which the reaction stops and the solution will only be subject to advection and diffusive decay (since it will remain below $\alpha$ forever).
Our next result is the same conclusion for the hyperbolic flows $\mathbf{(A1)}$ on $\rr^2$ and $\rr^3$ (we note that the eventual linear dynamics then decays as $e^{-At}$, which follows from rescaling the bounds in Corollary \ref{Cor:fast-decay} below).

\begin{thm}[{\bf Reaction-diffusion + hyperbolic flow}]\label{Thm:Quenching_in_Reaction_Diffusion}
Let $n$ solve \eqref{Reaction-Diffusion}  with $f$ from $\mathbf{(N1)}$, initial data $0\leq n_0\in L^1(\rr^d)$
and $d\in\{2,3\}$. There is  $A_0=A_0(||n_0||_{L^1},f)$ such that if $A\geq A_0$, then $\lim_{t\to\infty} || n(t,\cdot)||_{L^\infty}=0$.
\end{thm}


\subsection{Shear Flow Regime}
We next turn to \eqref{general} with the flow $\textbf{(A2)}$.  Our study of the PKS model with this type of advection  is  inspired by the above mentioned papers \cite{CKR00,KiselevZlatos} for advection-reaction-diffusion equations.  Similarly to them, we will consider here the PDE in the infinite channel $\rr\times\mathbb{T}^{d-1}$, and in its analysis we will not use the temporal scaling by factor $A$ that we applied in the case of the flow $\textbf{(A1)}$.


\medskip\noindent 
\textbf{Case $\mathbf{(A2)+(N1)}$.} 
This is the $d$-dimensional Patlak-Keller-Segel model with shear flow
\begin{equation}\label{eq:PKS_Shear}
\pa_t n+Au (\bx_-)\pa_{x_1} n=\de n - \na \cdot(n\na(-\Delta)^{-1} n), 
\qquad \bx=(x_1,\bx_-)\in \rr\times \mathbb{T}^{d-1},
\end{equation}
with $n(0,\bx)=n_0(\bx)\geq 0$ and $d\in\{2,3\}$.
Enhanced dissipation effects of shear flows on  $\mathbb{T}^2$ and $\mathbb{T}\times \rr$ (i.e., with closed streamlines), as well as their impacts on hydrodynamic stability of these flows, have been studied extensively in recent years (see, e.g.,  \cite{BMV14,BGM15I,BGM15II,BGM15III,BVW16, BCZ15,ElgindiCotiZelatiDelgadino18,Wei18,FengIyer19} and references therein). 
 The effects are caused by advection-induced phase mixing, which sends information to higher and higher Fourier modes and thereby amplifies diffusion.  
 
 On $\rr\times \mathbb{T}^{d-1}$, the role of mixing by shear flows is instead replaced by spreading \cite{CKR00,KiselevZlatos}.
%
In both cases, however, plateaus of the flows stand in the way of mixing/spreading by these flows, even when their amplitudes are large.  We therefore assume, as in \cite{CKR00}, that $u\in C^\infty(\mathbb T^{d-1})$ and the set
\begin{align}\label{def_P}
P_{u}:= \big\{x_2\in\mathbb T \, \big|\, 0=u'(x_2)=u''(x_2)=... \big\}
\end{align}
in the case $d=2$ is empty, with the corresponding assumption for $d=3$ being
 \begin{align}
P_{u}:=
 \big\{(x_2,x_3)\in\mathbb T^2 \, \big|\, \pa_{x_2}^\al\pa_{x_3}^{\beta}u(x_2,x_3)=0 \text{ for all $(\al,\beta)\in (\mathbb{N}\cup\{0\})^2\setminus \{(0,0)\}$} \big\}. \label{Plateau-3D}
\end{align}
These will ensure fast decay of solutions to the associated linear PDE 
\[
\pa_t \rho+Au(\bxm)\pa_{x_1}\rho=\de \rho
\]
(see Lemma \ref{L.3.1x} below), which will ultimately allow us to obtain our main result for shear flows.

\begin{thm}[{\bf PKS + shear flow}]\label{Thm:PKS_suppression_of_blow_up_TR}
Let $n$ solve \eqref{eq:PKS_Shear} with  initial data $0\leq n_0\in L^1(\rr\times \mathbb T^{d-1})\cap H^2(\rr\times \mathbb T^{d-1})$ 
and $d\in\{2,3\}$. If the  flow profile $u\in C^\infty(\mathbb T^{d-1})$ has no plateaus (i.e., $P_{u}=\emptyset$), then there is $A_0=A_0(||n_{0}||_{L^1\cap H^2},u)$ such that if $A\geq A_0$, then $\sup_{t\geq 0} \|n(t,\cdot)\|_{L^\infty}<\infty$. 
\end{thm}

%

We recall that, as mentioned above, shear flows $\textbf{(A2)}$ on the 3D domains $\mathbb{T}^3$ (as well as $\mathbb{T}\times \rr^2$) can only prevent blow-up of solutions with $\|n_0\|_1<8\pi$ \cite{BedrossianHe16}.  Theorem \ref{Thm:PKS_suppression_of_blow_up_TR} shows that such a limitation is not present in our fast spreading scenario of strong shear flows  $\textbf{(A2)}$ on $\rr\times \mathbb{T}^2$.

We also note that this result easily extends to infinite cylinders $\rr\times\Omega$ in place of $\rr\times \mathbb{T}^2$, with  smooth bounded domains $\Omega\subseteq \rr^{d-1}$ and Neumann boundary conditions for $n$ and $c$.

\medskip \noindent
{\bf Case $\mathbf{(A2)+(N2)}$}. 
This model was introduced  in \cite{CKR00}, where the authors proved that
shear flows with no plateaus in two dimensions do quench solutions with arbitrarily large $L^1(\rr\times\mathbb T)$ initial data (this result can be easily extended to three dimensions, see Lemma \ref{L.3.1x} and the proof of Theorem~\ref{Thm:Quenching_in_Reaction_Diffusion} below).  That of course settles this case in our analysis.  We also note that a sharp condition for quenching arbitrarily large initial data by a given shear flow profile $u$ on $\rr\times\mathbb T$ (with a sufficiently large amplitude $A$), which allows $P_{u}$ to contain any intervals shorter than a specific $f$-dependent threshold,  was obtained  in \cite{KiselevZlatos}.  

\medskip 
The paper is organized as follows. In Section \ref{Sec:Hyper}, we study the hyperbolic flow regime and prove Theorems \ref{Thm:PKS_suppression_of_blow_up}, \ref{Thm:PKS_long_time_behavior}, and \ref{Thm:Quenching_in_Reaction_Diffusion}.  In Section \ref{Sec:Shear} we consider the shear flow regime and prove Theorem \ref{Thm:PKS_suppression_of_blow_up_TR}.  We prove the relevant well-posedness and  regularity results in the appendix.

\section{Hyperbolic Flow Regime}\label{Sec:Hyper}
\subsection{Estimates on the Linear Evolution}\label{Sec:Green's function} 
Let us first derive the formula for the Green's function of the passive scalar equation (with  time re-scaled by $A$ relative to \eqref{PKS})
\begin{equation}\label{EQ:PS}
\pa_t \rho+(-\bxm,x_d)\cdot \na \rho= A^{-1} \de \rho, \qquad \bx=((d-1)\bx_-,x_d) \in \rr^d,
\end{equation}
with $\rho(0,\bx) =\rho_0(\bx)$ and $d\in\{2,3\}$.
We have the following lemma.

\begin{lem}\label{lem:Green}
Solutions to the diffusive passive scalar PDE \eqref{EQ:PS} have the following form.

\noindent
(i) In the 2D case, with $\bx=(x_1,x_2)\in \rr^2$:
\[
\rho(t,\bx)=\frac{A}{4\pi  \sinh {t}}\iint  \exp\left\{-\frac{|x_2e^{-t}- x_2'|^2}{2 A^{-1} (1-e^{-2t})}-\frac{|x_1e^t-x_1'|^2}{2 A^{-1} (e^{2t}-1)}
\right\} \rho_{0}(\bx')d\bx' .
\]

\noindent
(ii) In the 3D case, with $\bx=(x_1,x_2,x_3) \in \rr^3$:
\[
\rho(t,\bx)=\frac{A^{3/2}}{2^{5/2} \pi^{3/2}  \sqrt{e^{2t}-1\,} (1-e^{-t}) } \iint  \exp\left\{-\frac{|x_3e^{-t}-x_3'|^2}{2 A^{-1} ({1-e^{-2t}})} -\frac{|(x_1,x_2) e^{t/2}-(x_1',x_2')|^2}{4 A^{-1} (e^{t}-1)}
\right\} \rho_{0}(\bx')d\bx'.
\]
\ifx
\begin{align} \rho(t,\bx_-,x_3)=\frac{1}{(4\pi)^{3/2}\sqrt{\frac{e^{2t}-1}{2A} }\frac{1-e^{-t}}{A}}\iint  \exp\left\{-\frac{|x_3'-x_3e^{-t}|^2}{4\frac{1}{A}\big(\frac{1-e^{-2t}}{2}\big)}-\frac{|\bx_- e^{t/2}-\bx'_-|^2}{\frac{4}{A}(e^{t}-1)}
\right\} \rho_{0}(\bx')d\bx'.
\end{align}
\fi
\end{lem} 
\begin{proof}
\noindent
We begin with the 2D case. 
Since the vector field $(-x_1,x_2)$ is divergence-free, the equation \eqref{EQ:PS} can be rewritten as
\begin{align}
\pa_t \rho+\na\cdot((-x_1,x_2) \rho)= A^{-1} \de \rho.
\label{EQ:PS_div}
\end{align}
By direct calculation we observe that  \eqref{EQ:PS_div} respects the tensor product structure of the solution, i.e., if the initial data is of the form $\rho_0(x_1,x_2)=\rho_{1;0}(x_1)\rho_{2;0}(x_2)$, and $\rho_1(t,x_1)$ and $\rho_2(t,x_1)$ solve the following equations
\begin{align}
\pa_t\rho_1(t,x_1)-\pa_{x_1}(x_1\rho_1(t,x_1))= A^{-1} \de_{x_1}\rho_1(t,x_1),\qquad \rho_1(0,x_1)=\rho_{1;0}(x_1),\label{x_eq}\\
\pa_t\rho_2(t,x_2)+\pa_{x_2}(x_2\rho_2(t,x_2))= A^{-1}  \de_{x_2}\rho_2(t,x_2),\qquad \rho_2(0,x_2)=\rho_{2;0}(x_2),\label{y_eq}
\end{align}
then $\rho(t,x_1)=\rho_1(t,x_1)\rho_2(t,x_2)$ is a solution to the equation \eqref{EQ:PS}.

Next we solve the equations \eqref{x_eq} and \eqref{y_eq} explicitly.  The change of variables
\begin{align}\label{Change_of_Coordinate_y}
\rho_2(t,x_2)=: e^{-t} N\left(\tau,X_2 \right),\qquad (\tau,X_2)= \left( \frac{1-e^{-2t}}{2},\frac{x_2}{e^t} \right)
\end{align}
transforms \eqref{y_eq} into the heat equation 
\[
\pa_\tau N= A^{-1} \de_{X_2} N,\qquad N(0,X_2)=\rho_{2;0}(x_2),
\] 
which  is solved by
 \[N(\tau,X_2)=\frac{1}{\sqrt{ {4\pi} A^{-1} \tau}}\int_\rr\exp\left\{-\frac{|X_2-x_2'|^2}{4 A^{-1} \tau}\right\}\rho_{2;0}(x_2')dx_2'.
 \] 
From this and \eqref{Change_of_Coordinate_y} we have
\begin{align}\label{rho2}
\rho_2(t,x_2)=\frac{1}{\sqrt{  2 \pi A^{-1} (e^{2t}-1)}} \int_\rr \exp\left\{-\frac{|x_2e^{-t}-x_2'|^2}{2  A^{-1} (1-e^{-2t})}\right\}\rho_{2;0}(x_2')dx_2'.
\end{align}
Similarly,  the change of variables 
\[
\rho_1(t,x_1)=e^t M(\tau,X_1),\qquad (\tau,X_1)= \left( \frac{e^{2t}-1}{2},x_1e^{t} \right)
\]
transforms \eqref{x_eq} into the heat equation 
\[
\pa_\tau M= A^{-1} \de_{X_1} M,\qquad M(0,X_1)=\rho_{1;0}(x_1).\] 
Combining its explicit solution and the change of variables formula yields 
\begin{align}\label{rho1}
\rho_1(t,x_1)=\frac{1}{ \sqrt{ 2\pi  A^{-1} (1-e^{-2t}) }}\int_\rr \exp\left\{-\frac{|x_1e^t-x_1'|^2}{2 A^{-1} (e^{2t}-1)}\right\}\rho_{1;0}(x_1')dx_1'.
\end{align}  

If the initial data is already tensorized, i.e., $\rho_{0}(x_1,x_2)=\rho_{1;0}(x_1)\rho_{2;0}(x_2)$, the solution to \eqref{EQ:PS} is obtained by multiplying \eqref{rho1} and \eqref{rho2}, i.e.,
\[
\rho(t,\bx)=\frac{1}{ 2\pi A^{-1} (e^{t}-e^{-t})}\iint_{\rr^2} \exp\left\{-\frac{|x_2e^{-t}- x_2'|^2}{2 A^{-1} (1-e^{-2t})}-\frac{|x_1e^t-x_1'|^2}{2 A^{-1} (e^{2t}-1)}
\right\} \rho_{1;0}(x_1')\rho_{2;0}(x_2')d\bx' .
\]
Since $\rho_0$ can always be decomposed as $\sum_{i=1}^\infty n_i(x_1)m_i(x_2)$ and the PDE is linear, we obtain (i).

\medskip
We now turn to the 3D case. The main difference  is the equation for $\rho_1$, which now reads
\begin{equation}\label{x-eq-3D}
\pa_t\rho_1(t,x_1,x_2)-\frac{1}{2}\na\cdot((x_1,x_2) \rho_1(t,x_1,x_2)) = A^{-1} \de \rho_1(t,x_1,x_2), \qquad
\rho_1(0,x_1,x_2) =\rho_{1;0}(x_1,x_2),
\end{equation}
with $(x_1,x_2)\in\rr^2$.
After letting 
\begin{align*}
\rho_1(t,x_1,x_2)=:e^{t} M(\tau,\mathbf{X}),\qquad (\tau,\mathbf{X})= \left( \frac{e^{t}-1}{2},(x_1,x_2) e^{t/2} \right),
\end{align*} 
a direct calculation shows that $M$ solves the heat equation $\pa_\tau M=2 A^{-1} \de_{\mathbf{X}} M$. Hence \eqref{x-eq-3D} is solved by
\begin{align*}
\rho_1(t,x_1,x_2)=\frac{1}{4\pi A^{-1} (1-e^{-t})}\iint_{\rr^2} \exp\left\{-\frac{|(x_1,x_2) e^{t/2}-(x_1',x_2')|^2}{ 4 A^{-1}(e^{t}-1)} \right\}\rho_{1;0}(x_1',x_2')dx'_1dx_2'.
\end{align*} 
The rest of the proof is the same as in the 2D case, so we omit the details.
\end{proof}

Since both exponentials in the statement of Lemma \ref{lem:Green}  are less than $1$, we immediately obtain the following corollary.

\begin{cor}\label{Cor:fast-decay}
Solutions to  \eqref{EQ:PS} satisfy the  following  bounds:
\noindent
\begin{align}
(i) \ \ \mbox{In the 2D case:} \ \ \ ||\rho(t,\cdot)||_{L^\infty(\rr^2)}&\leq \frac{A}{\displaystyle \sinh t}||\rho(0,\cdot)||_{L^1(\rr^2)};\label{fast-decay}\\
(ii) \ \ \mbox{In the 3D  case:} \ \ \  ||\rho(t,\cdot)||_{L^\infty(\rr^3)}& \leq \frac{ A^{3/2}}{\displaystyle \sqrt{e^{2t}-1\,} (1-e^{-t})} ||\rho(0,\cdot)||_{L^1(\rr^3)}.&&&\label{fast-decay-3D}
\end{align}
\end{cor} 

We note that one can now use Schur's test or Young's inequality for integral operators to derive more general $L^p$--$L^q$ bounds. 

\subsection{Suppression of Blow-up in the Advective PKS Model}\label{Sec:PKS}
In this section, we prove Theorem \ref{Thm:PKS_suppression_of_blow_up}. We will mainly consider the 2D case, as the 3D case can  be treated similarly.

We begin by rescaling \eqref{PKS} in time to obtain
\begin{equation}\label{PKS_rescaled}
\pa_t n+(-\bxm,x_d)\cdot\na n = A^{-1} \de n-  A^{-1} \na \cdot(n \na c), \qquad -\de c= n,
\qquad \bx=((d-1)\bx_-,x_d)\in \rr^d,
\end{equation}
with $n(0,\bx) =n_0(\bx)\geq 0$ and  $d\in\{2,3\}$.

\begin{proof}[Proof of Theorem \ref{Thm:PKS_suppression_of_blow_up}]

\noindent 
\textbf{The 2D case}.
Let us first summarize the ideas of the proof in the 2D case. We first observe that due to the rescaling in \eqref{PKS_rescaled}, the nonlinear aggregation effects take place on a time scale of order $A$. For $t \ll A$, the solution to \eqref{PKS_rescaled} is well-approximated by the passive scalar dynamics \eqref{EQ:PS}, which will be verified via energy estimates. The second observation, due to Corollary \ref{Cor:fast-decay}, is that the passive scalar dynamics has strong spreading effect, i.e., the $L^\infty$ norm decays on a much shorter time scale $O(\log A)\ll A$. Therefore, both the nonlinear dynamics and the passive scalar dynamics decay quickly on this shorter time scale, so there will be no mass concentration anywhere on the plane. This rules out potential singularity formation for \eqref{PKS_rescaled}. 

In fact, we will show that 
\begin{align}
||n(t)||_2\leq 2||n_0||_2=:2B \qquad \forall t\geq 0,\label{goal1}
\end{align}
provided the amplitude $A$ is chosen large enough (depending on  $||n_0||_{L^1\cap L^\infty}$). To this end, let us consider the following two hypotheses for some time $t_\star\geq 0$ (they obviously hold for $t_\star=0$).

\smallskip
\noindent
\textbf{Hypothesis (a):} $||n(t_\star)||_2$ is no more than $B$:
\begin{align}
||n(t_\star)||_2 \leq  B. \label{H1}
\end{align}

\noindent 
\textbf{Hypothesis (b):} $||n(t)||_2$ is not too large (and $n$ exists) on the time interval $[0,t_\star]$:
\begin{align}
||n(t)||_2\leq 2B \qquad \forall t\leq t_\star.\label{H2}
\end{align}

 Our goal is now to prove the following two conclusions, given the above hypotheses and provided ${A}$ is large enough (depending only on  $||n_0||_{L^1\cap L^\infty}$).

\smallskip
\noindent
\textbf{Conclusion (a):} $||n(t)||_2$ is not too large (and $n$ exists) on the time interval $t\in [t_\star,t_\star+10\log A]$:
\begin{align}
||n(t)||_{2}\leq& 2B \qquad \forall t\in [t_\star,t_\star+10\log A]. \label{C1}
\end{align}

\noindent
\textbf{Conclusion (b):} $||n(t_\star +10\log A)||_2$ is no more than $B$:
\begin{align}
||n(t_\star+ & 10\log A)||_2\leq B .\label{C2}
\end{align}  
This will then immediately imply \eqref{goal1} for all large enough $A$, via iteration of the above for $t_\star=0,10\log A, 20\log A,\dots$, and the result will follow by Theorem \ref{thm:Local_Hs}.
So it remains to prove \eqref{C1} and \eqref{C2} for a solution $n$ to \eqref{PKS_rescaled} satisfying \eqref{H1} and \eqref{H2}.
We note that as long as $\|n(t)\|_2$ stays uniformly bounded, the same will be true for $\|n(t)\|_\infty$ and the total mass 
\begin{align}
M:=||n_0||_{1} = ||n(t)||_{1}\label{Mass}
\end{align}
will be preserved by Theorem \ref{thm:Local_Hs}.


\smallskip
\textbf{Step 1: Proof of  \eqref{C1}.}  
Let  $T_\star$ be the maximal time in  $[t_\star, t_\star+10\log A]$ such that
 \begin{align}
 \sup_{t\in[t_\star,T_\star]}||n(t)||_{2}\leq &4B\label{Bootstrap_hypothesis_4B}
\end{align}
(it exists by continuity of $||n(t)||_{2}$, see Theorems \ref{thm:local_Hs_0} and  \ref{thm:Local_Hs}).  We will then prove that in fact
\begin{align}
\sup_{t\in[t_\star,T_\star]}||n(t)||_{2}\leq &2B,\label{Bootstrap_conclusion_2B}
\end{align}  
which then shows that $T_\star$ must be $t_\star+10\log A$, 
completing the proof of  \eqref{C1}.
 
First note that \eqref{H2}, \eqref{Bootstrap_hypothesis_4B},  and Theorem \ref{thm:Local_Hs} imply
\begin{align}
||n||_{L_t^\infty([0,T_\star];L_\bx^\infty)} + |||\bx|\sqrt n||_{L_t^\infty([0,T_\star];L_\bx^2)} +  ||\na n||_{L_t^\infty([0,T_\star];L_\bx^2)}&<\infty.\label{Bootstrap_hypothesis_ML}
\end{align} 
These qualitative bounds 
will  justify integrations by parts below.

With these preparations in hand, we are ready to prove \eqref{Bootstrap_conclusion_2B}.  Integration by parts yields 
\begin{align}
\frac{d}{dt}||n||_2^2=&-\frac{2}{A}\int|\na n|^2d\bx +\frac{2}{A}\int \na n\cdot ( n\na (-\de)^{-1} n)d\bx +\int \bu\cdot \na (n^2) d\bx\\
=&-\frac{2}{A}\int|\na n|^2d\bx+\frac{1}{A}\int n^3 d\bx +\int \bu\cdot \na (n^2) d\bx\\
=&-\frac{2}{A}\int|\na n|^2d\bx+T_1+T_2,\label{T1T2}
\end{align}
with $\mathbf{u}=(-x_1,x_2)$.  To estimate $T_1$, we use the Gagliardo-Nirenberg interpolation inequality
to see that there is $C_{GN}$ such that
 \begin{align}
T_1
\leq&\frac{1}{A} \sqrt{C_{GN}} ||\na n||_2||n||_2^{2}\leq\frac{1}{A}||\na n||_2^2+\frac{C_{GN}}{A}  ||n||_2^4.\label{T1}
\end{align}
Next we estimate $T_2$ in \eqref{T1T2}. Since $\nabla\cdot \mathbf{u}=0$, we expect this term to be zero, via integration by parts. However, the vector field $\mathbf{u}$ is growing linearly at infinity, so the justification of this is nontrivial. 
From  \eqref{Bootstrap_hypothesis_ML} we have that $n\mathbf{u}\cdot\na n\in L^\infty_t([0,T_\star];L_\bx^1)$. The Lebesgue Dominated Convergence theorem and $\grad\cdot \mathbf{u}=0$ then show that 
\begin{align*}
T_2& 
=\lim_{R\rightarrow \infty} \int_{B(0,R)}\mathbf{u}\cdot \na (n^2)d\bx
 & =\lim_{R\rightarrow \infty} \int_{\partial B(0,R)}\mathbf{u}\cdot \frac{\bx}{|\bx|}n^2dS\leq \liminf_{R\rightarrow \infty} \int_{\partial B(0,R)} R n^2 dS,
\end{align*}
with $dS$  the arc-length element of the circle $\partial B(0,R)$. But since \eqref{Bootstrap_hypothesis_ML} and \eqref{Mass} show that
\begin{align}
\int_{0}^\infty \left(\int_{\pa B(0,R)} R n^2 dS \right) dR\leq  ||n||_\infty ||(1+|\bx|^2) n||_1<\infty,
\end{align} 
we see that indeed $T_2=0$.

Using the above estimates, from \eqref{T1T2} we obtain that
\begin{align}
\frac{d}{dt}||n(t)||_2^2\leq& 
-\frac{1}{A}||\na n(t) ||_2^2+ \frac{C_{GN}}{A}||n(t)||_2^4 \le
\frac{C_{GN}}{A}||n(t)||_2^4\qquad \forall t\in[t_\star,T_\star].\label{time_integral_L2} 
\end{align}
If  $A$ is chosen large enough so that $C_{GN}B^2                                                                                                                                                                                                                                                                                              \frac{10\log A}{A} \leq \frac 34$, this and \eqref{H1} yield
\begin{align}
||n(t)||_2\leq 2B\qquad \forall t\in[t_\star,T_\star].
\end{align}
This concludes the proof of  \eqref{Bootstrap_conclusion_2B} and hence of \eqref{C1}. 


\smallskip
\textbf{Step 2: Proof of \eqref{C2}.}
Let us first estimate the difference $n-\rho$, where $\rho$ is the  solution to \eqref{EQ:PS} on  $[t_\star,t_\star+ 10\log A]$ with initial data $\rho(t_\star)=n(t_\star)$. Similarly to \eqref{T1T2} and the following estimate on $T_2$, together with 
\[
2\int \na( n-\rho) \cdot ( n\na c)d\bx \leq 2||\na (n-\rho)||_2 ||n||_{4} ||\na c||_{4} \leq ||\na (n-\rho)||_2^2 +  ||n||_{4}^2 ||\na c||_{4}^2
\]  
in place of the estimate on $T_1$,
we obtain
\begin{align*}
\frac{d}{dt} ||n-\rho||_2^2\leq&
-\frac{1}{A}||\na (n-\rho)||_2^2+ \frac{1}{A}||n||_{4}^2 ||\na c||_{4}^2 \leq
\frac{1}{A}||n||_{4}^2 ||\na c||_{4}^2 .
\end{align*}
The Hardy-Littlewood-Sobolev inequality  yields $||\na c||_{4}\leq C_{HLS}||n||_{{4/3}}$.  Then by applying H\"older inequality and using Theorem \ref{thm:Local_Hs}, and then using \eqref{C1} and \eqref{Mass}, we obtain
\begin{align} \label{3.111}
\frac{d}{dt}||n-\rho||_2^2\leq &\frac{C_{HLS}}{A}||n||_{2}^2 ||n||_\infty ||n||_1 \leq \frac{4C_{HLS}}{A} B^2C_\infty M,
\end{align}
where $C_\infty=C_\infty(2B,\|n_0\|_{L^1\cap H^2})$ is from Theorem \ref{thm:Local_Hs} (and, crucially,  it does not depend on $A$).
Integrating this inequality on the time interval $[t_\star,t_\star+10\log A]$ and using $\rho(t_\star)=n(t_\star)$ shows that if $A$ is chosen large enough (depending on $||n_0||_{L^1\cap L^\infty}$), we have that
\begin{align}\label{L2deviation}
||n(t)-\rho(t)||^2_2\leq& \frac{40\log A}{A} C_{HLS} B^2C_\infty M\leq \frac{B^2}{4}\qquad\forall t\in[t_\star,t_\star+ 10\log A].
\end{align}
So $n$ stays close to $\rho$ on this time interval, and we now need to estimate the latter. 

Since \eqref{EQ:PS} preserves total mass of solutions, we have
\begin{align*} 
||\rho(t_\star+ 10\log A)||_2\leq||\rho(t_\star)||_1^{1/2}||\rho(t_\star+ 10\log A)||_\infty^{1/2} \leq M^{1/2}||\rho(t_\star+ 10\log A)||_\infty^{1/2}.
\end{align*}
Then \eqref{fast-decay} shows that for all large enough   $A$ we have
\begin{align}\label{2.222}
||\rho(t_\star+ 10\log A)||_2 &\leq M^{1/2}{A^{-4}}\leq \frac{B}{2}.
\end{align}
This and  \eqref{L2deviation} now yield \eqref{C2}, finishing the proof of the 2D case.

\vspace{5 mm}
\noindent
\textbf{The 3D case}.
This case is analogous to the 2D case, with a few minor adjustments.  In Step 1, the estimate on $T_1$ will now use both the Gagliardo-Nirenberg interpolation inequality and Young inequality to get
\[
\int n^3 d\bx \leq C_{GN}^{1/4} \|\nabla n\|_2^{3/2}\|n\|_2^{3/2} \leq \|\nabla n\|_2^2 + C_{GN}\|n\|_2^6.
\]
This then yields
\[ 
\frac{d}{dt}||n(t)||_2^2\leq  
-\frac{1}{A}||\na n(t) ||_2^2+ \frac{C_{GN}}{A}||n(t)||_2^6 \leq
\frac{C_{GN}}{A}||n(t)||_2^6\qquad \forall t\in[t_\star,T_\star]
\]
as before, so \eqref{C1} again follows for large enough $A$. 
In Step 2, the Hardy-Littlewood-Sobolev inequality yields $||\na c||_{4}\leq C_{HLS}||n||_{{12/7}}$ and we obtain
\begin{align*}
\frac{d}{dt}||n-\rho||_2^2\leq \frac{1}{A}||n||_{4}^2 ||\na c||_{4}^2 \leq &\frac{C_{HLS}}{A}||n||_{2}^{8/3} ||n||_\infty ||n||_1^{1/3} \leq \frac{8C_{HLS}}{A} B^{8/3} C_\infty M^{1/3}.
\end{align*}
So we again obtain the estimate \eqref{L2deviation} when $A$ is large enough.  The proof now concludes similarly to the 2D case, using \eqref{fast-decay-3D} instead of \eqref{fast-decay}.
\end{proof}

We conclude this subsection with the proof of Theorem \ref{Thm:PKS_long_time_behavior}.

\begin{proof}[Proof of Theorem \ref{Thm:PKS_long_time_behavior}]
 Let us fix an arbitrary small $\ep>0$.  Estimates \eqref{L2deviation} and \eqref{2.222} in the proof of Theorem \ref{Thm:PKS_suppression_of_blow_up}, together with \eqref{goal1}, guarantee that 
 $\sup_{t\geq 10\log A} ||n(t)||_{2}\leq \ep$ for all large enough $A$ (depending only on  $||n_0||_{L^1\cap H^2}$ and $\ep$).  Then, as in that proof, we obtain 
\begin{align*}
\frac{d}{dt}||n||_2^2\leq& -\frac{1}{A}||\na n||_2^2+\frac{C_{GN}}{A}||n||_2^6
\leq -\frac 1A \frac{||n||_{2}^6}{C_{GN}' ||n||_{3/2}^4}+\frac{C_{GN}}{A}||n||_2^6
\leq -\frac{||n||_2^6}{A}\left(\frac{1}{C_{GN}' ||n||_{3/2}^4}-C_{GN}\right),&&
\end{align*}
where we once more  used the Gagliardo-Nirenberg interpolation inequality.
Since we also have  $||n||_{3/2}\leq ||n||_1^{1/3}||n||_2^{2/3}$,
for all large $A$ and all $t\geq 10 \log A$ we have
 \begin{align}
\frac{d}{dt}||n(t)||_2^2\leq-\frac{||n(t)||_2^6}{A}\left(\frac{1}{C_{GN}'M^{4/3}||n(t)||_2^{8/3}}-C_{GN}\right)\leq -\frac{||n(t)||_2^4}{\ep^{1/3}A},
\end{align}
provided  we fix a small enough $\ep=\ep(M)\in(0,\frac 1{1000})$ (and then the lower bound on $A$ only depends on  $||n_0||_{L^1\cap H^2}$).
This and $10 \frac{\log A}A \leq \frac {10}e \leq (\ep^{-2}-1)\ep^{1/3}$ now imply  the bound
\[
\|n(t)\|_2 \leq \left(\|n(10\log A)\|^{-2}_2+ (\ep^{1/3}A)^{-1}(t-10\log A)\right)^{-1/2} \leq \left (1+(\ep^{1/3}A)^{-1}t \right)^{-1/2} \qquad \forall t \geq 10\log A
\]
for solutions to \eqref{PKS_rescaled}, which after rescaling in time becomes $\|n(t)\|_2 \leq (1+\ep^{-1/3}t)^{-1/2}$ for all $t \geq 10 \frac{\log A}A$ for solutions to \eqref{PKS}.  
This also yields
\[
\int_{B(0,t^{1/4})} n(t,\bx)d\bx \leq \frac 2{\sqrt 3} \sqrt\pi \|n(t)\|_2 t^{3/8}
\leq 3\ep^{1/6} t^{-1/8} \leq t^{-1/8}
\]
for these $t$, completing the proof.
\end{proof}

\subsection{Quenching in Advection-Reaction-Diffusion Equations}\label{Sec:ARD}
We now prove Theorem \ref{Thm:Quenching_in_Reaction_Diffusion}. 
We rescale \eqref{Reaction-Diffusion} to 
\begin{align}\label{ADR-rescale}
\pa_t n+(-\bxm,x_d)\cdot \na n=A^{-1}\de n+ A^{-1}f(n),\qquad \bx=(\bx_-,x_d)\in\rr^d,
\end{align}
with $n(0,\bx)=n_0(\bx)\geq 0$ and $d\in\{2,3\}$.

Since $f$ is Lipschitz, we have $f(n)\leq \beta n$ on $[0,1]$, for some $\beta\geq 0$.  The comparison principle then ensures that the solution to  \eqref{ADR-rescale} is bounded above by the solution to 
\begin{align}
\pa_t \wh n+(-\bxm,x_d)\cdot \na \wh n=A^{-1}\de \wh n+{\beta}A^{-1} \wh n,\qquad \wh n(0,\bx)=n_0(\bx).
\end{align}
But then $\wh n(t)=e^{\beta t/A}\rho(t)$, where $\rho$ solves \eqref{EQ:PS} with $\rho(0,\cdot)=n_0$.  From
\eqref{2.222} we see that for all large enough $A$ we have $\|\wh n(10\log A,\cdot)\|_\infty\leq \alpha$ (with $\alpha$ from $\mathbf{(N2)}$).  So after time $10\log A$, the solution $n$ to \eqref{ADR-rescale} solves \eqref{EQ:PS}, and the result now follows from $\|n(10\log A,\cdot )\|_{1}\leq e^{10\beta \log A/A}\|n_0\|_1<\infty$.

%
%

\section{Shear Flow Regime}\label{Sec:Shear}
Let us now consider the shear flow Advective PKS model  

\begin{align}\label{EQ:PKS_shear} 
\pa_t n+A u(\bx_-) \pa_{x_1} n+\na\cdot(\na c n)=\de n, \qquad \bx=(x_1,\bx_-)\in \mathbb{R}\times \mathbb{T}^{d-1},
\end{align}
with $n(0,\bx)=n_0(\bx)\geq 0$ and $d\in\{2,3\}$, and the corresponding passive scalar equation
\begin{align} \label{EQ:PS_shear} 
\pa_t \rho +Au(\bx_-)\pa_{x_1} \rho =\de \rho, 
\end{align}
with $\rho(0,\bx)=\rho_0(\bx)$.   Note that unlike in \eqref{PKS_rescaled}, time is not rescaled here.  Then we have the following replacement for Corollary \ref{Cor:fast-decay}, whose 2D case appears in \cite{CKR00}.

\begin{lem} \label{L.3.1x}
If $u$ has no plateaus (i.e., $P_{u}=\emptyset$), then for each $t>0$ there is $C_t<\infty$ (depending also on $u$) such that solutions to \eqref{EQ:PS_shear}  satisfy
\begin{align}\label{fast-decay_shear} 
||\rho(t,\cdot)||_{L^\infty(\rr\times\mathbb T^{d-1})}&\leq \frac{C_t}{A}||\rho_0||_{L^1(\rr\times\mathbb T^{d-1})}
\end{align}
\end{lem}

\begin{proof} 
%
Let us sketch here proof from \cite{CKR00} for $d=2$. 
If $\rho$ solves \eqref{EQ:PS_shear} and $\psi$ solves
\begin{align}
\pa_t \psi +Au(x_2)\pa_{x_1} \psi =& \pa_{x_2x_2} \psi ,\qquad \psi(0,\cdot)=\rho_0,\label{3.3_shear}
\end{align}
then
\begin{align}\label{PhiPsi_relation}
\rho(t,\bx)=\int_{-\infty}^\infty  \frac{1}{\sqrt{4\pi t}}\exp\left( -\frac{(x_1-y)^2}{4t}\right) \psi(t,y,x_2)\,dy
\end{align}
for all $t>0$. From H\"ormander's hypoellipticity theory \cite{Hormander67} (note that  the Lie algebra generated by operators  $\partial_{x_2}$ and $\partial_t+u(x_2)\partial_{x_1}$ consists of  vector fields $\partial_{x_2}, \partial_t+u(x_2)\partial_{x_1}, u'(x_2)\partial_{x_1}, u''(x_2)\partial_{x_1},\dots$, and hence spans the full tangent space $TM\cong \rr^3$ at each $(t,x_1,x_2)$ due to $P_{2D}=\emptyset$) we have that
for each $A\neq 0$, there is a smooth transition probability density $p_A$ such that
\[
\psi(t,x_1,x_2)= \int_{\rr\times\mathbb T} p_A(t,x_1-y_1,x_2,y_2) \rho_0(y_1,y_2) \, d(y_1,y_2).
\]
We clearly have
\begin{align}
p_A(t,x_1,x_2,y_2)= A^{-1} p_1\left(t,A^{-1} x_1,x_2,y_2\right),
\end{align}
so
\begin{align}
||\rho(t)||_{\infty}\leq ||\psi(t)||_{\infty}\leq A^{-1}||p_1(t)||_{\infty}||\rho_0||_{1}.
\end{align}
Therefore $C_t=||p_1(t)||_{\infty}$ works.

The  proof for $d=3$ is identical, with the operators $\partial_{x_2}$, $\partial_{x_3}$, and $\partial_t+u(x_2,x_3)\partial_{x_1}$ now generating the full tangent space $TM\cong \rr^4$ at each $(t,x_1,x_2,x_3)$ due to $P_u=\emptyset$.
\end{proof}

Now we are ready to prove Theorem \ref{Thm:PKS_suppression_of_blow_up_TR} (using Theorem \ref{thm:Local_shear} below in place of Theorem \ref{thm:Local_Hs}).

\begin{proof}[Proof of Theorem \ref{Thm:PKS_suppression_of_blow_up_TR}]
The proof is  similar to the proof of Theorem \ref{Thm:PKS_suppression_of_blow_up}, but without rescaling time and with \eqref{fast-decay_shear} instead of \eqref{fast-decay}.  The nonlinear aggregation effects now take place on a time scale of order $1$, so up to some time $\tau_\star \ll 1$, solutions to \eqref{EQ:PKS_shear} will be well-approximated by those to \eqref{EQ:PS_shear}. On the other hand, the latter solutions decay fast in $L^\infty $ on this shorter time scale when $A$ is large, so there will be no concentration of mass for solutions to \eqref{EQ:PKS_shear} either and singularities cannot form.  

We proceed as in the proof of Theorem \ref{Thm:PKS_suppression_of_blow_up}, with $n$ solving  \eqref{EQ:PKS_shear} and with $10\log A$ replaced by $0<\tau_\star\ll 1$ (to be specified) in Conclusions (a), (b).  In place of \eqref{T1T2} we now have (with $C_{GN}\geq 1$)
\begin{align}
\frac{d}{dt}||n||_2^2=&-2\int |\na n|^2d\bx+\int n^3d\bx 
\leq -||\na n||_2^2+ C_{GN} ||n||_2^4
\leq C_{GN} ||n||_2^4
\,\,\, \label{L2estimate-2D}
\end{align}
 in the 2D case.
Note that since the domain is quasi-onedimensional and the flow is bounded on it, it is straightforward to show that the flow term (denoted $T_2$ in \eqref{T1T2}) vanishes.
Also note that the Gagliardo-Nirenberg interpolation inequality used to estimate $T_1$ in \eqref{T1T2} extends to $\mathbb{R}\times \mathbb{T}$. This follows from applying it to $n$ periodically extended to $\mathbb{R}^2$ and multiplied by $\chi_{[-L,L]} (x_2)$, with $\chi_{[-L,L]}$ a smooth approximate characteristic function of $[-L,L]$, and then taking $L\to\infty$.  A similar extension of the relevant 3D Gagliardo-Nirenberg interpolation inequality holds in the 3D case.

So we obtain Conclusion (a) if we pick  $\tau_\star\leq \frac 34 C_{GN}^{-1}B^{-2}$ ($d=2$) or $\tau_\star\leq \frac {15}{32} C_{GN}^{-1}B^{-4}$ ($d=3$).
Next, in place of \eqref{3.111} we obtain (with $\rho$ solving \eqref{EQ:PS_shear}  and $\rho(t_\star)=n(t_\star)$)
\begin{align*}
\frac{d}{dt} ||n-\rho||_2^2\leq& -||\na (n-\rho)||_2^2+ ||n||_{2}^2 ||\na c||_{\infty}^2 
\leq C ||n||_{2}^2(||n||_1^2+||n||_\infty^2)\leq 4 C B^2(M^2+C_\infty^2) 
\end{align*}
for both $d=2,3$, using Lemma \ref{L.B.1x} below to estimate $||\na c||_{\infty}$. We therefore again obtain
\begin{align}\label{L2deviation_shear}
||n(t)-\rho(t)||^2_2\leq  \frac{B^2}{4} \qquad\forall t\in[t_\star,t_\star+\tau_\star],
\end{align}
provided $\tau_\star \leq \frac 1{16} C ^{-1}(M^2+C_\infty^2)^{-2}$ (with $C_\infty$ from Theorem \ref{thm:Local_shear}).
%
So if $\tau_\star>0$ is chosen small enough and then $A$ large enough (both depending only on $||n_{0}||_{L^1\cap L^\infty}$ and $u$), Conclusion (b) again follows from \eqref{L2deviation_shear} and
\begin{align*}
||\rho(t_\star+\tau_\star)||_2\leq ||\rho(t_\star+\tau_\star)||_1^{1/2}||\rho(t_\star+\tau_\star)||_\infty^{1/2} \leq 
C_{\tau_\star}^{1/2} A^{-1/2}||\rho(t_\star)||_{L^1}^{1/2}\leq C_{\tau_\star}^{1/2} A^{-1/2} M^{1/2} \leq \frac{B}{2}.
\end{align*} 
This finishes the proof.
\end{proof}


\begin{lem} \label{L.B.1x}
There is $C_d>0$ such that if $c=(-\de)^{-1} n$ on $\rr\times\mathbb{T}^{d-1}$, then
\begin{align}
||\na c||_{\infty}\leq& C_d(||n||_{1} +||n||_{{d+1}}).
\end{align} 
%
\end{lem}

\begin{proof}
If $\overline{n}(x_1):=\int n(x_1, \bx_-)d\bx_-$ and $\overline{c}(x_1):=\int c(x_1, \bx_-)d\bx_-$, then $\overline{c}''=\overline{n}$ on $\rr$, so 
\begin{align}
||\overline{c}'||_{\infty}\leq ||\overline{n}||_{1}= ||n||_{1}.
\end{align}
 Gagliardo-Nirenberg interpolation inequality (extended to $\rr\times\mathbb{T}^{d-1}$ as in the  proof of Theorem \ref{Thm:PKS_suppression_of_blow_up_TR} because $c-\overline{c}$  is mean-zero in $\bx_-$) shows that 
 \begin{align}
||\na (c-\overline{c})||_{\infty}\lesssim ||\na (c-\overline{c})||_{2}^{2/(d^2+d+2)} ||\na^2( c-\overline{c})||_{d+1}^{d(d+1)/(d^2+d+2)}.
\end{align}
Since $c-\overline{c}=(-\de)^{-1} (n-\overline{n})$ and is mean-zero in $\bx_-$, Poisson inequality yields
\[
||\na (c-\overline{c})||_{2}\lesssim ||n-\overline{n}||_{2}\leq ||n||_{L^2},
\]
while the Calder\' on-Zygmund inequality for the Riesz transform yields
\[
||\na^2( c-\overline{c})||_{d+1} \lesssim  ||n-\overline{n}||_{d+1} \leq ||n||_{d+1}.
\]
Hence
\begin{align}
||\na c||_{\infty}\lesssim ||n||_{1}+||n||_{2}^{2/(d^2+d+2)} ||n||_{d+1}^{d(d+1)/(d^2+d+2)},
\end{align}
and the result follows from H\"older and Young inequalities.
\end{proof}

\appendix

\section{Well-posedness and Regularity for the Hyperbolic Advective PKS}

We prove here well-posedness for \eqref{PKS} and \eqref{eq:PKS_Shear}.



\begin{thm}[Local well-posedness]\label{thm:local_Hs_0}
For $d\in\{2,3\}$, $s\geq 1$, and $0\leq n_0\in L^1(\rr^d)\cap H^s(\rr^d)$.
there is $t_0=t_0( ||n_0||_{L^2},A)>0$ such that a unique $H^s_\bx$-solution to \eqref{PKS} with $n(0,\bx)=n_0(\bx)$ exists on the time interval $[0,t_0]$.  Moreover, $||n(t,\cdot)||_{L^1}$ is constant, $||n(t,\cdot)||_{H^{s-1}}$  continuous, and $||n(t,\cdot)||_{H^{s}}$ bounded on $[0,t_0]$ (so solutions can blow up in finite time in $H^s_\bx$ only if they do in $L^2_\bx$).
%
\end{thm}

\begin{proof}
Let us only consider the 2D case, as the 3D case is similar.  Moreover, we will only derive relevant a-priori estimates on $||n(t,\cdot)||_{H^s}$ here, as local existence then follows via a standard approximation procedure (see, e.g.,  \cite{Biler1998,BlanchetEJDE06,BlanchetCarrilloMasmoudi08,BedrossianMasmoudi14,EganaMischler16}).  
Throughout the proof, we use the usual multi-index convention 
\[ 
|(\al_1,\al_2)|:=\al_1+\al_2\qquad\text{and}\qquad  
(\al_1,\al_2)\geq (\beta_1,\beta_2) \text{ iff  }\al_1\geq \beta_1 \text{ and }\al_2\geq \beta_2.
\]

We now perform the following change of variables (recall that $-\de c= n$)
\begin{align}
N(t, X_1,X_2):=n(t, x_1,x_2), \quad C(t, X_1,X_2):=c(t,x_1,x_2),\quad (X_1,X_2):=(x_1e^{At}, x_2e^{-At}),
\end{align}
and let $N_0:=n_0$. For all $p\in[1,\infty]$ we obviously have
\begin{align}\label{L^p_relation_x_X}
|| N(t)||_{L_{  X}^p}=||n(t)||_{L_{\bx}^p},
\end{align}
and \eqref{PKS} now becomes
\begin{align}
\pa_t N=e^{2At}\pa_{X_1X_1}N+e^{-2At}\pa_{X_2X_2}N -e^{2At}\pa_{X_1}(N\pa_{X_1}C)-e^{-2At}\pa_{X_2}(N\pa_{X_2}C),
\end{align} 
with $-e^{2At}\pa_{X_1X_1}C-e^{-2At}\pa_{X_2X_2}C=N$.  Since this is a divergence form PDE and $N\geq 0$, we obtain
\begin{align}
||n(t)||_{L_{\bx}^1}=||N(t)||_{L_{  X}^1}=||N_0||_{L_{  X}^1}=||n_0||_{L_\bx^1}=:M. \label{L_1_preserved}
\end{align} 

Next, similarly to \eqref{T1T2} we have
\begin{align}
\frac{d}{dt}\int N^2 d{  X} 
=-2e^{2At}\int |\pa_{X_1}N|^2 d{  X} -2e^{-2At}\int |\pa_{X_2}N|^2 d{  X} 
+  \int N^3  d{  X} ,
\end{align}
and Gagliardo-Nirenberg interpolation inequality $||N||_{3}^3\leq \sqrt{C_{GN}} ||\na N||_{2}||N||_{2}^2$
again yields
\begin{align}
\frac{d}{dt}||N(t)||_{2}^2 \leq -e^{2At} ||\pa_{X_1}N(t)||_2^2 -e^{-2At}||\pa_{X_2}N(t)||_2^2+C_{GN}e^{2At}||N(t)||_{2}^4.
\end{align}
This shows that 
there is $t_0=t_0(||n_0||_2,A)>0$ and
 $B_{0}=B_{0}(||n_0||_2,A)<\infty$ such that
\begin{align}
\sup_{t\in[0,t_0]} ||N(t)||_{2} \leq B_{0}. 
\label{L_2_estimate_appendix}
\end{align} 
\ifx
Now by the relation \eqref{L^p_relation_x_X}, we have that $||n(t)||_{L^2_{x_1,x_2}}$ is continuous and bounded on a short time interval, whose length depends on the initial data. We apply the $L^4$-energy estimate, together with an application of the Gagliardo-Nirenberg-Sobolev inequality to obtain that
\begin{align}
\frac{d}{dt}&\frac{1}{4}\int N^4dX\\
=&-\frac{3}{4}e^{2At}\int |\pa_{X_1}N^2|^2dX-\frac{3}{4}e^{-2At}\int |\pa_{X_2}N^2|^2dX-\frac{3}{4}\int N^4(e^{2At}\pa_{X_1X_1}C+e^{-2At}\pa_{X_2X_2}C)dX\\
=&-\frac{3}{4}e^{2At}\int |\pa_{X_1}N^2|^2dX-\frac{3}{4}e^{-2At}\int |\pa_{X_2}N^2|^2dX+\frac{3}{4}\int N^5 dX\\
\leq&-\frac{1}{4}e^{2At}\int |\pa_{X_1}N^2|^2dX-\frac{1}{4}e^{-2At}\int |\pa_{X_2}N^2|^2dX+Be^{6At}||N||_{2}^{8}.
\end{align}\fi
Similarly to \eqref{3.111}, Hardy-Littlewood-Sobolev inequality yields $||\na c||_{4}\leq C_{HLS}||n||_{{4/3}}$, 
and explicit representation of the Green's function of the Laplacian on $\rr^2$, we have that 
\begin{align}
||\na C(t)||_{4}\leq e^{At}||\na c(t)||_{4}\leq C_{HLS} e^{At}(||n(t)||_{1}+||n(t)||_{2}) 
\leq C_{HLS} e^{At}(M+ B_{0})
\label{na_c_L_4_appendix} 
\end{align}
for $t\in[0,t_0]$.
Furthermore,  Calderon-Zygmund inequality shows that for $t\in[0, t_0]$,
\begin{align}
||\na^2 C(t)||_{4}\leq e^{2At}||\na^2 c(t)||_{4}\leq 
C_{CZ} e^{2At}||n(t)||_{4} = C_{CZ} e^{2At}||N(t)||_{4}.
\label{na2_c_L_4_appendix}
\end{align}

We now use these estimates and Gagliardo-Nirenberg interpolation inequality $||N||_4\lesssim ||\na N||_2^{1/2} || N||_2^{1/2}$ to control $||\na N||_{2}$ via (with $B',B''$ some universal  constants)
\begin{align}
\frac{d}{dt}&\frac{1}{2} ||\na N||_{2}^2 = \frac{d}{dt} \frac{1}{2}\sum_{|\al|=1}||D^\al_X N||_{2}^2\\
=&-\sum_{|\al|=1} e^{2At}||\pa_{X_1} D^\al_X N||_2^2-\sum_{|\al|=1} e^{-2At}||\pa_{X_2} D^\al_X N||_2^2 \\
&+\sum_{|\al|=1} e^{2At}\int \pa_{X_1} D^\al_X N D^\al_X   (N \pa_{X_1} C) dX 
+\sum_{|\al|=1} e^{-2At}\int \pa_{X_2}D^\al_X N D^\al_X   (N \pa_{X_2} C) dX\\
\leq& -e^{-2At}||\na^2 N||_2^2
+ \sum_{j=1}^2 \sum_{|\al|=1} e^{(-1)^{j-1}2At} || \pa_{X_j}D^\al_X N||_2 \left( ||D^\al_X  \pa_{X_j} C||_4||N||_4+||\pa_{X_j}C||_4||D_X^\al N||_4 \right)  \\
\leq & -e^{-2At}||\na^2 N||_2^2
+B' e^{2At} || \na^2 N||_2 \left(||\na    N||_2 ||N||_2+ e^{At} (M+B_{0}) ||\na^2 N||_2^{1/2} ||\na  N||_2^{1/2} \right)  \\
\leq& -\frac{1}{2}e^{-2At}||\na^2 N||_2^2+ B'' e^{18At} \left( 1+M^4  + B_{0}^4 \right)|| \na N||_2^2.
\end{align}
This shows that $||\na N(t)||_{2}$ is bounded 
on the time interval $[0,t_0]$,
and so
\begin{align}
\sup_{t\in[0,t_0]} \max \left\{ ||N(t)||_{H^1}, 
||\na C(t)||_{L^4} 
\right\} \leq B_1
\label{B.11} 
\end{align} 
for some $B_{1}=B_{1}(||n_0||_{L^1\cap H^1},A)<\infty$. 

We can now obtain recursively $H^r$-bounds for $r=2,3,\dots,s$.  Assume that 
\begin{align}\label{H_s-1_assumption}
\sup_{t\in[0,t_{0}]} \max \left\{ ||N(t)||_{H^{r-1}},  ||\na C(t)||_{L^4} \right\} \leq B_{r-1}.
\end{align}
Then
\begin{align}
\frac{d}{dt}\frac{1}{2}& \sum_{|\al|=r}||D^\al_X N||_{2}^2 = -\sum_{|\al|=r} e^{2At}||\pa_{X_1} D^\al_X N||_2^2-\sum_{|\al|=r} e^{-2At}||\pa_{X_2} D^\al_X N||_2^2 \\
& -\sum_{|\al|=r} e^{2At}\int  D^\al_X N \pa_{X_1}D^\al_X  (N\pa_{X_1} C) dX 
-\sum_{|\al|=r} e^{-2At}\int    D^\al_X N \pa_{X_2}D^\al_X (N\pa_{X_2} C) dX,
\label{d/dt D_X_al N_2}
\end{align}
and we also have for $j=1,2$,
\begin{align}
\bigg|\int& D_X^\al N\pa_{X_j}D_X^\al(N\pa_{X_j}C)dX\bigg| \\
&\leq 2^{|\alpha|} \sum_{ \ell\leq \al} ||\pa_{X_j}D_X^{\al-\ell}N||_2 ||D_X^\al N||_4 ||\pa_{X_j}D_X^\ell C||_4+ 
 2^{|\alpha|} \sum_{ { \ell\leq \al}} ||D_X^\al N||_2||D_X^{\al-\ell}N||_4||\pa_{{X_j}X_j}D_X^\ell C||_4. 
\end{align}
Similarly to  \eqref{na2_c_L_4_appendix}, we obtain (with some $r$-dependent $B'$ and for all $|\ell|\leq r$)
\begin{align}
||\na^{|\ell|+1} C(t)||_{4}\leq & B' e^{(\ell+1)At}||\na^{|\ell|-1} N(t)||_{4} \leq B' e^{(\ell+1)At} ||\na^{|\ell|} N(t)||_{2}^{1/2}||\na^{|\ell|-1} N(t)||_{2}^{1/2},
\end{align}
and we also use the Gagliardo-Nirenberg interpolation inequality $||\na^{k}N||_4\lesssim ||\na^{k+1} N||_2^{1/2} ||\na^k N||_2^{1/2}$.
Combining the above estimates with \eqref{d/dt D_X_al N_2} and \eqref{H_s-1_assumption}, we see that (with some $r$-dependent $B''$)
\begin{align}
\frac{d}{dt}  ||\na^r N(t)||_{2}^2 \leq - e^{-2At} ||\na^{r+1} N(t)||_{2}^2 + B''e^{B''At} (1+ B_{r-1}^4) (1+ ||\na^r N(t)||_{2}^2).
\end{align}
This yields \eqref{H_s-1_assumption} for $r=3,4,\dots$, with $B_{r}=B_{r}(||n_0||_{H^r},A,B_{r-1})=B_{r}(||n_0||_{L^1\cap H^r},A)<\infty$.
This and \eqref{d/dt D_X_al N_2} also show continuity of $||N(t)||_{H^{s-1}}$ and boundedness of $||N(t)||_{H^{s}}$ on $[0,t_0]$.
%
\end{proof}

If $s\geq 2$, Theorem \ref{thm:local_Hs_0} and Sobolev embedding yield an $L^\infty$ bound on solutions, which depends on $n_0$, $A$, and $t$.  We used such a bound in the proof of Theorem \ref{Thm:PKS_suppression_of_blow_up}, but that proof requires it to not depend on $A$.   We will therefore need to prove the following result instead.

\begin{thm}[$A$-independent $L^\infty$ bound]\label{thm:Local_Hs}
For $d\in\{2,3\}$ and $0\leq n_0\in L^1(\rr^d)\cap H^2(\rr^d)$ with $ |\bx|^2 n_0(\bx)\in {L^1}(\rr^d)$,
assume that the unique $H^2_\bx$-solution $n$ to \eqref{PKS} from Theorem \ref{thm:local_Hs_0} satisfies
\begin{align}
||n||_{L_t^\infty([0,T];L_\bx^2)}\leq C_{L^2} \label{L2_bound_criteria}
\end{align}
for some $T,C_{L^2}<\infty$  (here $T$ is assumed to be smaller than the maximal time of existence for $n$).  Then $||n(t,\cdot)||_{L^1}$ is constant, $||n(t,\cdot)||_{H^{1}}$ continuous, and $||n(t,\cdot)||_{H^{2}}$ bounded on $[0,T]$, and
\begin{align}
||\, |\bx|^2 n||_{L^\infty_t([0,T]; L_\bx^1)} &\leq C_0=C_0(n_0, A,T)<\infty,\label{Moment_n} \\
||n||_{L_t^\infty([0,T]; L_\bx^\infty)} &\leq C_\infty= C_\infty( C_{L^2} , ||n_0||_{L^1\cap H^2})<\infty.\label{Linfty_n}
\end{align}
\end{thm}

\begin{proof}
We will again only consider $d=2$.
The first three claims are from Theorem \ref{thm:local_Hs_0}, and its proof also shows that
\begin{align}
||n||_{L_t^\infty([0,T]; H_\bx^2)}\leq C_{H^2}=C_{H^2}(C_{L^2}, ||n_0||_{L^1\cap H^2},A, T)<\infty. \label{qualitative_bound}
\end{align}
By Sobolev embedding we also have (cf.~\eqref{Linfty_n}) 
 \begin{align}
||n||_{L_t^\infty([0,T]; L_\bx^\infty)}\leq C_{L^\infty} = C_{L^\infty}(C_{L^2}, ||n_0||_{L^1\cap H^2}, A, T)<\infty. \label{qualitative_bound_2}
\end{align}
 Furthermore, Gagliardo-Nirenberg interpolation inequality and Hardy-Littlewood-Sobolev inequality show that
\begin{align}
||\na c||_\infty\lesssim ||n||_3^{3/5} ||\na c||_4^{2/5}  \lesssim ||n||_3^{3/5} ||n||_{4/3}^{2/5} \lesssim ||n||_1+||n||_4 \lesssim ||n||_1+||n||_\infty,
\label{B.1a}
\end{align}
so
\begin{align}
||\na c||_{L^\infty_t([0,T];L_\bx^\infty)}   \leq C_{\na c} = C_{\na c}(C_{L^2}, ||n_0||_{L^1\cap  H^2}, A,T)<\infty.\label{qualitative_bound_3}
\end{align}
Note that once we obtain \eqref{Linfty_n}, we will in fact have $C_{\na c} = C_{\na c}(C_{L^2}, ||n_0||_{L^1\cap  H^2})$ here.

We now prove \eqref{Moment_n}, via an argument similar to \cite{BlanchetEJDE06}.
Multiply  \eqref{PKS} by a sequence of $C^\infty_c$-functions $\varphi_k$ approximating $|\bx|^2$, with uniformly bounded $\na^2 \varphi_k$ and $|\bx|\,|\na \varphi_k(\bx)|\leq C(1+\varphi_k(\bx))$, and recall that  the Green's function for the Laplacian on $\rr^2$ is $\frac{1}{2\pi}\log |\cdot|$, to obtain for $t\in[0,T]$,
\begin{align*}
\frac{d}{dt}\int_{\rr^2} n(t,\bx) \varphi_k(\bx)d\bx
=&\int_{\rr^2} \de n(t,\bx)  \varphi_k d\bx - A\int_{\rr^2} \na\cdot((-x_1,x_2)n(t,\bx)) \varphi_k(\bx)d\bx. \\
& +\frac{1}{2\pi}\iint_{\rr^2\times\rr^2}  \na_\bx \cdot \left(\frac{\bx-\mathbf{y}}{|\bx-\mathbf{y}|^2}n(t,\mathbf{y})n(t,\bx)\right) \varphi_k(\bx) d\bx d\mathbf{y}
\end{align*}
Applying integration by parts and a standard symmetrization trick, we see that
\begin{align}
\frac{d}{dt}\int_{\rr^2} n(t,\bx) \varphi_k(\bx)d\bx
=&\int_{\rr^2}n(t,\bx)\de  \varphi_k(\bx)d\bx+ {A}\int_{\rr^2} n(t,\bx) (-x_1,x_2)\cdot\na  \varphi_k(\bx)d\bx\\
&-\frac{1}{4\pi  }\iint_{\rr^2\times\rr^2}\frac{(\bx-\mathbf{y})\cdot(\na  \varphi_k(\bx)-\na \varphi_k(\mathbf{y}))}{|\bx-\mathbf{y}|^2}n(t,\bx)n(t,\mathbf{y})d\bx d\mathbf{y}, \qquad \label{moment_general_form}
\end{align}
and then integrating in time and taking $k\to\infty$ yields
\begin{align}
\int_{\rr^2}n(t,\bx)|\bx|^2d\bx \leq \int_{\rr^2}n_0(\bx)|\bx|^2d\bx +  \left({4M}-\frac{M^2}{2\pi } \right)t + 2A \int_0^t \int_{\rr^2} n(s,\bx)|\bx|^2d\bx \, ds\qquad \forall t\in[0,T].
\end{align}
Gronwall's inequality  then shows \eqref{Moment_n}.

Finally, to prove \eqref{Linfty_n}, we will first estimate $L^p$-norms of $n$ for $p=2^j$ ($j\in \mathbb{N}$) via
\begin{align}
\frac{d}{dt}\frac 1{2p} \int n^{2p}d\bx=&-\frac{2p-1}{p^2} \int |\na n^p|^2d\bx+\frac{2p-1}{p}\int n^{p} \na n^p\cdot \na c d\bx+ A\int n^{2p-1}\mathbf{u}\cdot \na n d\bx\\
\leq&-\frac 1{2p} ||\na n^p||_2^2+p||n^p||_2^2||\na c||_\infty^2+ A\int n^{2p-1}\mathbf{u}\cdot \na nd\bx.\label{n_L2p_1x}
\end{align} 
From \eqref{qualitative_bound_2}, $|\bx|^{-1}\mathbf u(\bx)\in L^\infty(\rr^2)$, $|\bx|\sqrt n \in L_t^\infty([0,T]; L_\bx^2)$ (see \eqref{Moment_n}), and $\na n\in L_t^\infty([0,T]; L_\bx^2)$ we have that $n^{2p-1}\mathbf{u}\cdot \na n\in L_t^\infty([0,T]; L_\bx^1)$.
So for each $t\in[0,T]$ we obtain
\begin{align}
\int_{\rr^2}n^{2p-1}\mathbf{u}\cdot\na nd\bx =&\lim_{R\rightarrow\infty} \int_{B(0,R)}\mathbf{u}\cdot\na n^{2p}d\bx 
=\lim_{R\rightarrow\infty} \int_{\pa B(0,R)}\mathbf{u}\cdot\frac{\bx}{|\bx|} n^{2p}dS .
\end{align}
It now follows from  \eqref{qualitative_bound_2} and \eqref{Moment_n} that
\begin{align*}
\left| \int_{\rr^2}n^{2p-1}\mathbf{u}\cdot\na nd\bx \right|\leq  A ||n||_\infty^{2p-1} \liminf_{R\to\infty} \int_{\pa B(0,R)} R n \, dS  =0,
\end{align*}
and then \eqref{n_L2p_1x}, Nash inequality, and \eqref{B.1a} show for some universal $C>0$,
\begin{align}
\frac{d}{dt} ||n||_{2p}^{2p}
\leq- \frac{|| n^p||_2^4}{C ||n^p||_1^2} + Cp^2 ||n^p||_2^2 ( ||n||_1^2+||n||_4^2) 
= \frac{||n||_{2p}^{2p}} C \left( - \frac{||n||_{2p}^{2p}} {||n||_p^{2p}} +(Cp)^2 \left( ||n_0||_1^2+||n||_4^2 \right) \right).
\label{n_L2p_1}
\end{align} 

Now we apply the Moser-Alikakos iteration (see, e.g.,  \cite[Lemma 3.2]{CalvezCarrillo06},  \cite[Appendix]{KiselevXu15}, or \cite{Kowalczyk05}) to get \eqref{Linfty_n}.  Let $a_j:=||n||_{L_t^\infty([0,T];L^{2^j}_\bx) }$.
When $j=1$ (so $p=2$),  the above estimate and \eqref{L2_bound_criteria} show
\[
a_{2} \leq C_{L^4}= C_{L^4} (C_{L^2},||n_0||_{L^1\cap L^4}) <\infty.
\]
When $j=2,3,\dots$, after using $||n_0||_{2^{j+1}}\leq ||n_0||_{L^1\cap L^\infty} := ||n_0||_1+||n_0||_\infty$ we instead obtain
\[
a_{j+1}\leq \max\left\{ (\tilde C2^j)^{1/2^j} a_j,  ||n_0||_{L^1\cap L^\infty}  \right\},
\]
with $\tilde C:=\max\{C ( ||n_0||_1+C_{L^4}),1\}$.  
Then for all $j\geq 2$ we have
\[
a_{j}\leq \max \left\{C_{L^4}, ||n_0||_{L^1\cap L^\infty} \right\} \prod_{k=2}^\infty (\tilde C2^k)^{1/2^k} 
=:C_\infty (C_{L^2},||n_0||_{L^1\cap H^2})
\]
(recall that $H^2(\rr^2)\subseteq L^\infty(\rr^2)$ by Sobolev embedding),
yielding \eqref{Linfty_n} with this $C_\infty$.
\end{proof}

Finally, we note that the above results and proofs easily extend to the shear flow case \eqref{eq:PKS_Shear}, with no requirement on the second spatial moment of $n$ needed here because $\bf u$ is bounded.

\begin{thm}\label{thm:Local_shear}
For $d\in\{2,3\}$, $u\in C^2(\mathbb T^{d-1})$, and $0\leq n_0\in L^1(\rr\times \mathbb T^{d-1})\cap H^2(\rr\times \mathbb T^{d-1})$,
 there is $t_0=t_0( ||n_0||_{L^2},A)>0$ such that there is a unique $H^2_\bx$-solution $n$ to \eqref{eq:PKS_Shear} with $n(0,\bx)=n_0(\bx)$ on the time interval $[0,t_0]$, and $n$ remains unique on any time interval $[0,T]$ on which $||n(t,\cdot)||_{L^2}$ remains bounded.  Moreover, if
\begin{align}
||n||_{L_t^\infty([0,T];L_\bx^2)}\leq C_{L^2} \label{L2_bound_criteria_shear}
\end{align}
for some $T,C_{L^2}<\infty$  (here $T$ is assumed to be smaller than the maximal time of existence for $n$),  then $||n(t,\cdot)||_{L^1}$ is constant, $||n(t,\cdot)||_{H^{1}}$ continuous, and $||n(t,\cdot)||_{H^{2}}$ bounded on $[0,T]$, and
\begin{align}
||n||_{L_t^\infty([0,T]; L_\bx^\infty)} \leq C_\infty= C_\infty( C_{L^2} , ||n_0||_{L^1\cap H^2})<\infty.\label{Linfty_n_shear}
\end{align}
\end{thm}

\bibliographystyle{abbrv}
\bibliography{nonlocal_eqns,JacobBib,SimingBib}

\end{document}